\documentclass[11pt]{article}
\usepackage[english]{babel}
\usepackage{amsmath}
\usepackage{amsthm}
\usepackage{amssymb}
\usepackage[dvips]{color}
\newtheorem{theorem}{Theorem}[section]
\newtheorem{definition}[theorem]{Definition}
\newtheorem{problem}[theorem]{Problem}
\newtheorem{lemma}[theorem]{Lemma}
\newtheorem{proposition}[theorem]{Proposition}

\newtheorem{remark}[theorem]{Remark}

\newcommand{\rr}{\mathbb{R}}

\title{\bf The general theory of superoscillations and supershifts in several variables}
\date{}
\author{ F. Colombo
 \footnote{Politecnico di Milano, Dipartimento di Matematica, Via E. Bonardi, 9 20133 Milano, Italy, \ {\tt fabrizio.colombo@polimi.it, \, irene.sabadini@polimi.it,\, stefano.pinton@polimi.it}},\ S. Pinton$^{*}$, \
 I. Sabadini$^{*}$,\
  D.C. Struppa\footnote{The Donald Bren Presidential Chair in Mathematics, Chapman University, Orange, USA, \ {\tt struppa@chapman.edu}}
}

\begin{document}
\maketitle

\begin{abstract}
In this paper we describe a general method to generate
superoscillatory functions of several variables starting from a superoscillating sequence of one variable.
Our results are based on the study of suitable infinite order differential operators on holomorphic functions
with growth conditions of exponential type, where additional constraints are required when dealing with infinite
order differential operators whose symbol is a function that is holomorphic in some open set, but not necessarily entire.
The results proved for the superoscillating sequence in several variables are
extended to sequences of supershifts  in several variables.
 \end{abstract}

\vskip 1cm


\section{Introduction }
Superoscillating functions are  band-limited functions that  can oscillate faster than
their fastest Fourier component. Physical phenomena associated with superoscillatory functions
have been known for a long time
for example in antennas theory see \cite{TDFG},
and in the context of weak values in quantum mechanics, see \cite{aav}.
In more recent years there has been  a wide interest
in the theory of superoscillating functions
and of supershifts, a notion that generalizes the one of superoscillations, and that was introduced in the literature in
 order to study the evolution of superoscillations as initial data
 of the Schr\"odinger equation of other field equations, like Dirac or Klein-Gordon equations.

An introduction to superoscillatory functions in one variable and some applications to Schr\"odinger evolution of superoscillatory initial data can be found in \cite{acsst5}.
Superoscillatory functions in several variables have been rigorously defined and
 studied  in \cite{JFAA} and in \cite{ACJSSST22} where we have initiated also the theory of supershifts in more then one variable.
The aim of this paper is to remove the restrictions in \cite{JFAA,ACJSSST22}
 and to obtain a very general theory of superoscillations and supershifts.

\medskip

Our results are directed to a general audience of physicists, mathematicians, and engineers, and
our main tool is the theory of infinite order differential operators acting on spaces of holomorphic functions.
The literature on superoscillations is quite large, and without claiming completeness we have tried to mention some of the most relevant (and recent) results. Papers
\cite{ABCS19}-\cite{acsst5}, \cite{uno}, \cite{Jussi}, \cite{due}, \cite{Pozzi} and \cite{PeterQS} deal with the issue of permanence of superoscillatory behavior when evolved under a suitable Schr\"odinger equation; papers \cite{berry2}-\cite{b4}, \cite{kempf1}-\cite{kempf2HHH} and \cite{sodakemp} are mostly concerned with the physical nature of superoscillations, while papers \cite{QS3}, \cite{AShushi}, \cite{AOKI}-\cite{aoki3}, \cite{Talbot}-\cite{CSSYgenfun} develop in depth the mathematical theory of superoscillations. Finally we have cited \cite{acsst5} as a good reference for the state of the art on the mathematics of superoscillations until 2017, and the {\it Roadmap on Superoscillations} \cite{Be19}, where the most recent advances in superoscillations and their applications to technology are well explained by the leading experts in this field.

In this paper we extend the results in \cite{ACJSSST22}
considering analytic functions in one variable $G_1,\ldots, G_d$, $d\geq2$, whose
Taylor series at zero have radius of convergence grater than or equal to 1. Thus we define general superoscillating functions
of several variables as expressions of the form
$$
F_n(x_1,x_2,\ldots,x_d):=\sum_{j=0}^n Z_j(n,a)e^{ix_1 G_1(h_j(n))}e^{ix_2G_2(h_j(n))} \ldots e^{ix_d G_d(h_j(n))}
$$
where $Z_j(n,a)$, $j=0,...,n$, for $n\in \mathbb{N}_0$ are suitable coefficients of a superoscillating function in one variable as we will see in the sequel.
We will give conditions on the functions
$G_1,\ldots, G_d$ in order that
$$
\lim_{n\to \infty}F_n(x_1,x_2,\ldots,x_d)=e^{ix_1 G_1(a) }e^{ix_2 G_2(a) }\ldots e^{ix_d G_d(a)},
$$
so that, when $|a|>1$, $F_n(x_1,x_2,\ldots,x_d)$ is superoscillating.
 Moreover, we shall also treat the case of sequences that admit a supershift in $d\geq 2$ variables.

The paper is organized in four sections including the introduction.
Section 2 contains the preliminary material on superoscillations, the relevant function spaces and their topology, and the study of the continuity of some infinite order differential operators acting on such spaces. Section 3 is the main part of the paper and contains the definition of superoscillating functions in $d\geq 2$ variables as well as some results. Section 4 discusses the notion of supershift in this framework.

\section{Preliminary results on infinite order differential operators}

We begin this section with some preliminary material on superoscillations and  supershifts in one variable. Then we introduce and study some infinite order differential operators that will be of crucial importance
to define and study superoscillations and supershifts in several variables.

\begin{definition}\label{SUPOSONE}
We call {\em generalized Fourier sequence}
a sequence of the form
\begin{equation}\label{basic_sequenceq}
f_n(x):= \sum_{j=0}^n Z_j(n,a)e^{ih_j(n)x},\ \ \ n\in \mathbb{N},\ \ \ x\in \mathbb{R},
\end{equation}
where $a\in\mathbb R$, $Z_j(n,a)$ and $h_j(n)$
are complex and real valued functions of the variables $n,a$ and $n$, respectively.
The sequence (\ref{basic_sequenceq})
 is said to be {\em a superoscillating sequence} if
 $\sup_{j,n}|h_j(n)|\leq 1$ and
 there exists a compact subset of $\mathbb R$,
 which will be called {\em a superoscillation set},
 on which $f_n(x)$ converges uniformly to $e^{ig(a)x}$,
 where $g$ is a continuous real valued function such that $|g(a)|>1$.
\end{definition}
The classical Fourier expansion is obviously not a superoscillating sequence since its frequencies are not, in general, bounded.

In the recent paper \cite{newmet} we enlarged the class of superoscillating functions, with respect to the existing literature,
 and we
 solved the following problem.
\begin{problem}\label{gsgdfaA}
Let $h_j(n)$ be a given set of points in $[-1,1]$, $j=0,1,...,n$, for $n\in \mathbb{N}$ and let $a\in \mathbb{R}$ be such that $|a|>1$.
Determine the coefficients $X_j(n)$ of the sequence
$$
f_n(x)=\sum_{j=0}^nX_j(n)e^{ih_j(n)x},\ \ \ x\in \mathbb{R}
$$
in such a way that
$$
f_n^{(p)}(0)=(ia)^p,\ \ \ {\rm for} \ \ \ p=0,1,...,n.
$$
\end{problem}
\begin{remark}
The conditions $f_n^{(p)}(0)=(ia)^p$  mean that the functions $x\mapsto e^{iax}$ and $x\mapsto f_n(x)$ have the same derivatives at the origin, for
$p=0,1,...,n$,  and therefore the same Taylor polynomial of order $n$.

\end{remark}

\begin{theorem}[Solution of Problem \ref{gsgdfaA}]\label{exsolution}
Let $h_j(n)$ be a given set of points in $[-1,1]$, $j=0,1,...,n$ for $n\in \mathbb{N}$ and let $a\in \mathbb{R}$ be such that $|a|>1$.
If $h_j(n)\not= h_i(n)$, for every $i\not=j$,
then the coefficients $X_j(n,a)$ are uniquely determined and given by
\begin{equation}\label{sxplsolut}
X_j(n,a)=
\prod_{k=0,\  k\not=j}^n\Big(\frac{h_k(n)-a}{h_k(n)-h_j(n)}\Big).
\end{equation}
As a consequence, the sequence
$$
f_n(x)=\sum_{j=0}^n\prod_{k=0,\  k\not=j}^n\Big(\frac{h_k(n)-a}{h_k(n)-h_j(n)}\Big)\ e^{ix h_j(n)},\ \ \ x\in \mathbb{R}
$$
solves Problem \ref{gsgdfaA}.
Moreover, when the holomorphic extensions of the functions $f_n$ converge in $A_1$, we have
$$
\lim_{n\to\infty}f_n(x)=e^{iax},\ \ {\rm for \ all} \ \ x\in \mathbb{R}.
$$
\end{theorem}

Our approach to the study of superoscillatory functions in one or several
variables makes use of infinite order differential operators. Such operators naturally act on spaces of holomorphic functions. This is the reason for which we consider the holomorphic
 extension to entire functions of the sequence
$f_n(x)$ defined in (\ref{SUPOSONE}) by
replacing the real variable $x$ with the complex variable $\xi$.
For the sequences of entire functions we shall consider, a natural notion of convergence is the convergence
 in the space $A_1$ or in the space $A_{1,B}$ for some real positive constant $B$ (see the following definition and considerations).

\begin{definition}\label{sect2-def1}

The space $A_1$ is the  complex algebra of entire functions such that there exists $B>0$ such that
\begin{equation}\label{1B}
\sup\limits_{\xi\in \mathbb{C}} \big(|f(\xi)|\, \exp(-B|\xi|)\big)  <+\infty.
\end{equation}
\end{definition}
The space $A_1$ has a rather complicated topology, see e.g. \cite{BG}, since it is a linear space obtained via an inductive limit. For our purposes, it is enough to consider, for any fixed $B>0$,
 the set $A_{1,B}$ of functions $f$ satisfying \eqref{1B}, and to observe that
$$
\|f\|_{B}:=\sup\limits_{\xi\in \mathbb{C}} \big(|f(\xi)|\, \exp(-B|\xi|)\big)
$$
defines a norm on $A_{1,B}$, called the $B$-norm. One can prove that $A_{1,B}$ is a Banach space with respect to this norm.

Moreover, let $f$ and a sequence $(f_n)_n$ belong to $A_1$; $f_n$ converges to $f$ in $A_1$ if and only if there exists $B$ such that $f,f_n\in A_{1,B}$ and
\[
\lim\limits_{n\rightarrow\infty}\sup\limits_{\xi\in \mathbb{C}} \big|f_n(\xi)-f(\xi)\big|e^{-B|\xi|}=0.
\]
 With these notations and definitions we can make the notion of continuity explicit (see \cite{aoki3}):

 A linear operator $\mathcal{U}:\ A_1\to A_1$ is continuous if and only if for any $B>0$ there exists $B'>0$ and $C>0$ such that
 \begin{equation}\label{contA1}
 \mathcal{U}(A_{1,B})\subset A_{1,B'} \ {\rm and}\qquad
 \| \mathcal{U}(f)\|_{B'} \leq C\| f\|_B, \qquad {\rm for\ any\  } f\in A_{1,B}.
 \end{equation}

The following result, see Lemma \ref{BETOA1} in \cite{AOKI}, gives a characterization of the functions in $A_1$ in terms of the coefficients appearing in their Taylor series expansion.
\begin{lemma}\label{BETOA1}
The entire function
$$
f(\xi)=\sum_{j=0}^\infty a_j\xi^j
$$
belongs to $A_1$
if and only if there exists $C_f>0$ and $b>0$ such that
$$
|a_j|\leq C_f \frac{b^j}{\Gamma(j+1)}.
$$
\end{lemma}
\begin{remark}
To say that $f\in A_1$ means that $f\in A_{1,B}$ for some $B>0$. The computations in the proof of Lemma \ref{BETOA1} in \cite{AOKI}, show that $b=2eB$, and that we can choose $C_f=\| f \|_B$.
\end{remark}
 We now define two infinite order differential operators that will be used to study superoscillatory functions and supershifts in several variables. We shall denote by $\underline x$ the vector $(x_1,\ldots, x_d)$ in $\rr^d$.

 \begin{proposition}\label{CONTU_bis}
Let $d$ be a positive integer and let $R_\ell\in\rr_{+}\cup\{\infty\}$ for any $\ell=1,\dots,\, d$. Let $(g_{1,m}), \dots,\, (g_{d,m})$ be $d$ sequences of complex numbers such that
\begin{equation}\label{entireee_tris}
\lim \sup_{m\to\infty}|g_{\ell,m}|^{1/m}=\frac 1{R_\ell},\ \ for \ \ \ell=1,\dots,\, d.
\end{equation}
Let $x_1,\ldots, x_d\in \mathbb{R}$. Denote by $D_\xi:=\frac{\partial}{\partial \xi}$
the derivative operator with respect to the auxiliary complex variable $\xi$. We define the formal operator:
\begin{equation}\label{op1}
\mathcal{U}(x_1,x_2,\ldots, x_d,D_\xi):=\sum_{m=0}^\infty \frac{1}{m!}\sum_{k_1=0}^{\infty}\left(\sum_{k_2=0}^{k_1}\ldots \sum_{k_m=0}^{k_{m-1}} y_{k_m}y_{k_{m-1}-k_m}\ldots y_{k_1-k_2}\right) \frac {D_{\xi}^{k_1}}{i^{k_1}}
\end{equation}
where we have set
$$
y_p:=ix_1g_{1,p}+\ldots+ix_dg_{d,p},\ \ {\rm for} \ \ p=1,\ldots r \ \ {\rm with} \ \ r \in \mathbb{N}.
$$
Then, setting
 $$
 R:=\min_{\ell=1,\ldots, d} R_\ell,
 $$
 for any real value $0<B<\frac R{4e}$, the operator $\mathcal{U}(x_1,\ldots, x_d,D_\xi):A_{1,B} \to A_{1,4eB}$ is continuous for all $\underline x\in\mathbb R^d$.
\end{proposition}

\begin{proof} Let us consider $f\in A_{1,B}$; then we have
\[
\begin{split}
\mathcal{U}&(x_1,\ldots ,x_d,D_\xi)f(\xi)=\sum_{m=0}^\infty \frac{1}{m!}\sum_{k_1=0}^{\infty}\left(\sum_{k_2=0}^{k_1}\ldots \sum_{k_m=0}^{k_{m-1}} y_{k_m}y_{k_{m-1}-k_m}\ldots y_{k_1-k_2}\right) \frac {D_{\xi}^{k_1}}{i^{k_1}} f(\xi)
\\
&
=\sum_{m=0}^\infty \frac{1}{m!}\sum_{k_1=0}^{\infty}\left(\sum_{k_2=0}^{k_1}\ldots \sum_{k_m=0}^{k_{m-1}} y_{k_m}y_{k_{m-1}-k_m}\ldots y_{k_1-k_2}\right) \sum_{j=k_1}^\infty a_j \frac{j!}{(j-k_1)!}\xi^{j-k_1}
\\
&
=\sum_{m=0}^\infty \frac{1}{m!}\sum_{k_1=0}^{\infty}\left(\sum_{k_2=0}^{k_1}\ldots \sum_{k_m=0}^{k_{m-1}} y_{k_m}y_{k_{m-1}-k_m}\ldots y_{k_1-k_2}\right) \sum_{j=0}^\infty a_{j+k_1} \frac{(j+k_1)!}{j!}\xi^{j}.
\end{split}
\]
Taking the modulus we get
\[
\begin{split}
|\mathcal{U}&(x_1,\ldots ,x_d,D_\xi)f(\xi)|
\\
&
\leq \sum_{m=0}^\infty \frac{1}{m!}\sum_{k_1=0}^{\infty}\left(\sum_{k_2=0}^{k_1}\ldots \sum_{k_m=0}^{k_{m-1}} |y_{k_m}| |y_{k_{m-1}-k_m}|\ldots |y_{k_1-k_2}|\right) \sum_{j=0}^\infty |a_{j+k_1}| \frac{(j+k_1)!}{j!}\xi^{j}.
\end{split}
\]
and Lemma \ref{BETOA1} gives the estimate on the coefficients $a_{j+k_1}$
$$
|a_{j+k_1}|\leq C_f\frac{b^{j+k_1}}{\Gamma(j+k_1+1)}.
$$
where $b=2eB$. Using the well known inequality
$
(a+b)!\leq 2^{a+b}a!b!
$
we also have
$$
\left(j+k_1\right)!\leq 2^{j+k_1} j!k_1!
$$
so we get
$$
|\mathcal{U} (x_1,\ldots ,x_d,D_\xi)f(\xi)|\leq \sum_{m=0}^\infty \frac{1}{m!}\sum_{k_1=0}^{\infty}\left(\sum_{k_2=0}^{k_1}\ldots \sum_{k_m=0}^{k_{m-1}} |y_{k_m}| |y_{k_{m-1}-k_m}|\ldots |y_{k_1-k_2}|\right) \times
$$
$$
\times C_f\sum_{j=0}^\infty \frac{b^{j+k_1}}{\Gamma(j+k_1+1)}\frac{ 2^{j+k_1} k_1!j!}{j!}|\xi|^{j}.
$$
Now we use the Gamma function estimate
\begin{equation}\label{GAMab2}
\frac{1}{\Gamma(a+b+2)}\leq \frac{1}{\Gamma(a+1)}\frac{1}{\Gamma(b+1)}
\end{equation}
to separate the series, and we have
$$
\frac{1}{\Gamma(j-\frac{1}{2}+k_1-\frac{1}{2}+2)}\leq  \frac{1}{\Gamma(j+\frac{1}{2})}\frac{1}{\Gamma(k_1+\frac{1}{2})}
$$
and so
$$
|\mathcal{U}(x_1,\ldots ,x_d,D_\xi)f(\xi)|\leq  C_f\sum_{m=0}^\infty \frac{1}{m!}\sum_{k_1=0}^{\infty}\left(\sum_{k_2=0}^{k_1}\ldots \sum_{k_m=0}^{k_{m-1}} |y_{k_m}| |y_{k_{m-1}-k_m}|\ldots |y_{k_1-k_2}|\right)
\times
$$
$$
 \times \frac{(k_1)! (2b)^{k_1}}{\Gamma(k_1+\frac{1}{2})}
 \sum_{j=0}^\infty\frac{1}{\Gamma(j+\frac{1}{2})}(2b|\xi|)^{j}.
$$
Now observe that the latter series satisfies the estimate
$$
\sum_{j=0}^\infty\frac{1}{\Gamma(k+\frac{1}{2})}(2b|\xi|)^{j}\leq Ce^{4b|\xi|}
$$
where $C$ is a positive constant,
because of the properties of the Mittag-Leffler function; moreover, the series
\begin{equation}
\label{swrCCC}
\sum_{m=0}^\infty \frac{1}{m!}\sum_{k_1=0}^{\infty}\left(\sum_{k_2=0}^{k_1}\ldots \sum_{k_m=0}^{k_{m-1}} |y_{k_m}| |y_{k_{m-1}-k_m}|\ldots |y_{k_1-k_2}|\right) \frac{(k_1)! (2b)^{k_1}}{\Gamma(k_1+\frac{1}{2})}
\end{equation}
 is convergent and is bounded by a positive real constant $C_{\underline x,G_1,\ldots, G_d}$.
 In fact, using Stirling formula for the Gamma function, we have
 $$
 m!\sim \sqrt{2\pi m}\, e^{-m}m^m, \ \ {\rm for} \  \ \ m\to \infty
 $$
 and then we deduce
 \begin{equation}\label{gam1e12}
 \frac{\Gamma(m+1)}{\Gamma(m+1/2)}\sim \frac{\sqrt{2\pi \, m}\, e^{-m}m^m}{\sqrt{2\pi (m-1/2)}\, e^{-(m-1/2)}\, (m-1/2)^{(m-1/2)}}\sim\sqrt{m-1/2}, \ \ {\rm for} \  \ \ m\to \infty
 \end{equation}
 so that
 $$
 \frac{k_1!}{\Gamma(k_1+\frac{1}{2})}
 \sim\sqrt{k_1-1/2}, \ \ {\rm for} \  \ \ k_1\to \infty.
 $$
 Now observe that the series (\ref{swrCCC}) has positive coefficients
 and so it converges  if and only if the series
 $$
\sum_{m=1}^\infty \frac{1}{m!}\sum_{k_1=1}^{\infty}\left(\sum_{k_2=0}^{k_1}\ldots \sum_{k_m=0}^{k_{m-1}} |y_{k_m}| |y_{k_{m-1}-k_m}|\ldots |y_{k_1-k_2}|\right) (2b)^{k_1}\sqrt{k_1-1/2}
 $$
 converges. Given an absolutely convergent series $\sum_{p=0}^\infty a_p$, then its $m$-th power can be computed by means of the Cauchy product as follows:
\begin{equation}\label{cp}
\left(\sum_{p=0}^{\infty} a_p\right)^m=\sum_{k_1=0}^{\infty}\sum_{k_2=0}^{k_1}\ldots\sum_{k_m=0}^{k_{m-1}} a_{k_m}a_{k_{m-1}-k_m}\ldots a_{k_1-k_2}.
\end{equation}
 Using the inequality:
 $$\sqrt{k_1-\frac 12}\leq k_1\leq k_m+(k_{m-1}-k_m)+\ldots+(k_1-k_2)\leq (k_m+2)\cdot(k_{m-1}-k_m+2)\cdot\cdot\cdot\cdot\cdot(k_1-k_2+2),$$
 where $k_1\geq k_2\geq\dots\geq k_m$, we deduce that there exists a positive constant $C_{\underline x, G_1,\ldots, G_d}$ such that the following chain of inequalities hold: 
 \[
 \begin{split}
& \sum_{m=1}^\infty \frac{1}{m!}\sum_{k_1=1}^{\infty}\left(\sum_{k_2=0}^{k_1}\ldots \sum_{k_m=0}^{k_{m-1}} |y_{k_m}| |y_{k_{m-1}-k_m}|\ldots |y_{k_1-k_2}|\right) (2b)^{k_1}\sqrt{k_1-1/2}
 \\
 &
\leq \sum_{m=1}^\infty \frac{1}{m!}\sum_{k_1=1}^{\infty}\left(\sum_{k_2=0}^{k_1}\ldots \sum_{k_m=0}^{k_{m-1}} |y_{k_m} (k_m+2)(2b)^{k_m}| |y_{k_{m-1}-k_m} (k_{m-1}-k_m+2)(2b)^{k_{m-1}-k_m}|\times \right.\\
&\qquad\qquad\qquad\quad\quad\quad\quad\quad\quad\quad\quad\quad\quad\quad \ldots\times   |y_{k_1-k_2} (k_2-k_1+2)(2b)^{k_1-k_2}|\Bigg)
 \\
 &
 = \sum_{m=1}^\infty \frac{1}{(m)!}\left [\sum_{p=0}^{\infty} |y_p| (p+2)(2b)^p \right ]^m \leq \sum_{m=1}^\infty \frac{1}{(m)!}
\left [\sum_{p=1}^\infty |x_1|(p+2)(2b)^p|g_{1,p}| \right.+\\
&\qquad\qquad\qquad\quad\quad\quad\quad\quad\quad\quad\quad\quad\quad\quad \ldots+|x_d| (p+2)(2b)^p |g_{d,p}| \Bigg ]^m
\leq C_{\underline x, G_1,\ldots, G_d}
 \end{split}
 \]
 where for the equality we used \eqref{cp}, while the last inequality follows by the assumption
 $$
 B<\frac R{4e}$$
  which implies $2b<R$. From the previous estimate we have that the series (\ref{swrCCC}) converges for all $x_1,\ldots, x_d \in \mathbb{R}$.
 So we finally have
\begin{equation}\label{estCCC}
 |\mathcal{U}(x_1,\ldots, x_d,D_\xi)f(\xi)|\leq  C_f \, C_{\underline x,G_1,\ldots, G_d}\, C\,e^{4b|\xi|},\ \ \ \underline x\in \mathbb{R}^d, \ \ \xi\in \mathbb{C}.
\end{equation}
Recalling that $b=2eB$, the estimate \eqref{estCCC} implies that $\mathcal{U}(x_1,\ldots, x_d,)f\in A_{1, 8eB}$, in fact
$$
 |\mathcal{U}(x_1,\ldots, x_d,D_\xi)f(\xi)|\,e^{-8eB|\xi|}\leq  C_f \, C_{\underline x,G_1,\ldots, G_d}\, C\ \ \ \underline x\in \mathbb{R}^d, \ \ \xi\in \mathbb{C}.
$$
Moreover, we deduce that the $8eB$-norm satisfies the estimate
$$
\|\mathcal{U}(x_1,\ldots, x_d,D_\xi )f\|_{8eB} \leq C_f \, C_{\underline x,G_1,\ldots, G_d}\, C=  C_{\underline x,G_1,\ldots,G_d}\, C \| f\|_B.
$$
Thus $\mathcal{U}(x_1,\ldots, x_d,D_\xi):A_{1,B} \to A_{1,8eB}$ is continuous for all $\underline x\in\mathbb R^d$.
 \end{proof}

\begin{remark}\label{supernova}
Whenever we fix a compact subset $K\subset \rr^d$, we have that, for any $\underline x\in K$, the constants $C_{\underline x,G_1,\ldots, G_d}$ appearing in the proof of the previous theorem are bounded by a constant which depends only on $K$ and $G_1,\ldots, G_d$. Moreover, if $R_\ell=\infty$ for any $\ell=1,\dots, d$, the continuity of the operator $\mathcal U(x_1,\dots, x_d,D_\xi)$ holds for any $B>0$ and the proof of the previous theorem shows that $\mathcal U(x_1,\dots, x_d,D_\xi)$ is a continuous operator in $A_1$.
\end{remark}


\begin{proposition}\label{contsupshif_bis}
Let $d$ be a positive integer and let $R_\ell\in\rr_{+}\cup\{\infty\}$ for any $\ell=1,\dots,\, d$. Let $(g_{1,m}), \dots,\, (g_{d,m})$ be $d$ sequences of complex numbers such that
\begin{equation}\label{entireee_bis}
\lim \sup_{m\to\infty}|g_{\ell,m}|^{1/m}=\frac 1{R_\ell},\ \ for \ \ \ell=1,\dots,\, d.
\end{equation}
We define the formal  operator
\begin{equation}\label{inftysh_bis}
\mathcal{V}(x_1,\dots , x_d, D_\xi):=\sum_{m_1=0}^\infty  g_{1,m_1}\dots \sum_{m_d=0}^\infty g_{d,m_d}x_1^{m_1}  \dots  x_d^{m_d} \frac{1}{i^{m_1+\dots +m_d}}D_\xi^{m_1+\dots +m_d},
\end{equation}
where $x_1,\dots,\, x_d\in\mathbb R,\ \xi\in\mathbb{C}$. Then, for any real value $B>0$, the operator $\mathcal{V}(x_1,\dots, x_d, D_\xi): A_{1,B}\to A_{1,8eB}$ is continuous whenever $|x_\ell |< \frac{R}{4eB}$ for any $\ell=1,\ldots, d$ where $R:=\min_{\ell=1,\dots, d}R_{\ell}$.
\end{proposition}
\begin{proof}
We apply the operator $\mathcal{V}(x_1,\dots, x_d,D_\xi)$ to a function $f$ in $A_{1,B}$ for $|\underline x|< \frac{R}{4eB}$. We have
\[
\begin{split}
\mathcal{V}&(x_1,\dots, x_d,D_\xi)f(\xi)=\sum_{m_1=0}^\infty  g_{1,m_1}\dots \sum_{m_d=0}^\infty g_{d,m_d}x_1^{m_1}\dots x_d^{m_d} \frac{1}{i^{m_1+\dots +m_d}}D_\xi^{m_1+\dots +m_2}f(\xi)
\\
&
=\sum_{m_1=0}^\infty  g_{1,m_1}\dots \sum_{m_d=0}^\infty g_{d,m_d}x_1^{m_1}\dots x_d^{m_d} \frac{1}{i^{m_1+\dots +m_d}}D_\xi^{m_1+\dots +m_d}\sum_{j=0}^\infty a_j\xi^j
\\
&
=\sum_{m_1=0}^\infty  g_{1,m_1}\dots \sum_{m_d=0}^\infty g_{d,m_d}x_1^{m_1}\dots x_d^{m_d} \frac{1}{i^{m_1+\dots +m_d}}\times \\
&
\times \sum_{j=m_1+\dots+m_d}^\infty a_j \frac{j!}{(j-(m_1+\dots+m_d))!}\xi^{j-(m_1+\dots+m_d)}
\\
&
=\sum_{m_1=0}^\infty  g_{1,m_1}\dots \sum_{m_d=0}^\infty g_{d,m_d}x_1^{m_1}\dots x_d^{m_d} \frac{1}{i^{m_1+\dots +m_d}}\sum_{k=0}^\infty a_{m_1+\dots +m_d+k} \frac{(m_1+\dots+m_d+k)!}{k!}\xi^{k}.
\end{split}
\]
We then have
\[
\begin{split}
|\mathcal{V}(x_1,\dots,x_d,D_\xi)f(\xi)|\leq
\sum_{m_1=0}^\infty  |g_{1,m_1}|\dots &\sum_{m_d=0}^\infty |g_{d,m_d}||x_1|^{m_1}\dots |x_d|^{m_d}\times\\
&\times \sum_{k=0}^\infty |a_{m_1+\dots +m_d+k}| \frac{(m_1+\dots +m_d+k)!}{k!}|\xi|^{k}
\end{split}
\]
and using the estimate in Lemma \ref{BETOA1}
$$
|a_{m_1+\dots m_d+k}|\leq C_f \frac{b^{m_1+\dots m_d+k}}{\Gamma(m_1+\dots+m_d+k+1)},
$$
where $b=2eB$, we get
$$
|\mathcal{V}(x_1,\dots,x_d,D_\xi)f(\xi)|\leq
\sum_{m_1=0}^\infty  |g_{1,m_1}|\dots \sum_{m_d=0}^\infty |g_{d,m_d}||x_1|^{m_1}\dots |x_d|^{m_d}  \times
$$
$$
\times
C_f \sum_{k=0}^\infty\frac{b^{m_1+\dots +m_d+k}}{\Gamma(m_1+\dots +m_d+k+1)}\frac{(m_1+\dots +m_d+k)!}{k!}|\xi|^{k}.
$$
With the estimates
$$
(m_1+\dots +m_d+k)!\leq 2^{m_1+\dots +m_d+k}(m_1+\dots +m_d)!k!
$$
and
$$
\frac{1}{\Gamma(m_1+\dots +m_d-\frac{1}{2}+k-\frac{1}{2}+2)}\leq \frac{1}{\Gamma(m_1+\dots +m_d+\frac{1}{2})}\frac{1}{\Gamma(k+\frac{1}{2})}
$$
we separate the series
$$
|\mathcal{V}(x_1,\dots, x_d,D_\xi)f(\xi)|\leq
\sum_{m_1=0}^\infty  |g_{1,m_1}|\dots \sum_{m_d=0}^\infty |g_{d,m_d}||x_1|^{m_1}\dots |x_d|^{m_d}
\times
$$
$$
\times
 \sum_{k=0}^\infty C_f b^{m_1+\dots +m_d+k}
 \frac{1}{\Gamma(m_1+\dots +m_d+\frac{1}{2})}\frac{1}{\Gamma(k+\frac{1}{2})}\frac{2^{m_1+\dots +m_d+k}(m_1+\dots +m_d)!k! }{k!}|\xi|^{k}.
$$
Finally we get
$$
|\mathcal{V}(x_1,\dots,x_d,D_\xi)f(\xi)|\leq C_f
\sum_{m_1=0}^\infty  |g_{1,m_1}|\dots \sum_{m_d=0}^\infty |g_{d,m_d}|(2b|x_1|)^{m_1}\cdots (2b|x_d|)^{m_d}
  \times
$$
$$
\times \frac{(m_1+\dots +m_d)!}{\Gamma(m_1+\dots +m_d+\frac{1}{2})} \sum_{k=0}^\infty \frac{1}{\Gamma(k+\frac{1}{2})}(2b|\xi|)^{k}.
$$
Using (\ref{gam1e12}) we have
$$
\frac{(m_1+\dots +m_d)!}{\Gamma(m_1+\dots +m_d+\frac{1}{2})}\sim\sqrt{m_1+\dots +m_d-1/2}, \ \ {\rm for} \  \ \ m_1+\dots +m_d\to \infty ,
$$
and, moreover, $\sqrt{m_1+\dots +m_d-1/2}\leq m_1\cdots m_d$ if $m_\ell \geq 2$ for any $\ell=1,\dots, d$. Since $|x_\ell|< \frac{R}{4eB}$ for any $\ell=1,\ldots, d$ and $b=2eB$, the series
$$
\sum_{m_\ell=1}^\infty m_\ell |g_{\ell,m_\ell}|(2b|x_\ell|)^{m_\ell}
$$
converges to a constant which depends on $x_\ell \in \mathbb{R}$. Thus there exist constants $C_{x_{\ell}}$ such that
$$
|\mathcal{V}(x_1,\dots, x_d,D_\xi)f(\xi)|\leq C_f C_{x_1}\dots C_{x_d} (2b|\xi|)e^{2b|\xi|}\leq C_f C_{x_1,\dots,x_d} e^{4b|\xi|}
$$
from which, recalling that $C_f=\|f\|_B$, we deduce
$$
\|\mathcal{V}(x_1,\dots, x_d,D_\xi)f\|_{8eB}\leq  C_{x_1,\dots, x_d} \|f\|_B.
$$
We conclude that the operator $\mathcal{V}(x_1,\dots, x_d, D_\xi): A_{1,B}\to A_{1,8eB}$ is continuous.
\end{proof}
\begin{remark}\label{rnova1}
Whenever we fix a compact subset
$$
K\subset \{\underline x\in\rr^d:\, |x_\ell|<\frac{R}{4eB}\,\textrm{for any $\ell=1,\ldots, d$}\},
$$
 we have that, for any $\underline x\in K$, the constants $C_{x_\ell}$'s, appearing in the proof of the previous theorem are bounded by a constant which depends only on $K$. Moreover, if $R_\ell=\infty$ for any $\ell=1,\dots, d$, the continuity of the operator $\mathcal V(x_1,\dots, x_d,D_\xi)$ holds to be true for any $\underline x\in\rr^d$ and the proof of the previous theorem shows that $\mathcal V(x_1,\dots, x_d,D_\xi)$ satisfies the conditions in \eqref{contA1}. Thus we conclude that $\mathcal V(x_1,\dots, x_d,D_\xi)$ is a continuous operator in $A_1$.
\end{remark}

\section{Superoscillating functions in several variables}

We recall some preliminary definitions related to superoscillating functions in several variables.

\begin{definition}[Generalized Fourier sequence in several variables]
For $d\in \mathbb{N}$ such that $d\geq 2$, we assume that
$(x_1,...,x_d)\in \mathbb{R}^d$. Let
 $(h_{j,\ell}(n))$,  $j=0,...,n$  for $ n\in \mathbb{N}_0$, be  real-valued sequences for $\ell=1,...,d$.
We call {\em generalized Fourier sequence in several variables}
a sequence of the form
\begin{equation}\label{basic_sequence_sev}
F_n(x_1,\ldots ,x_d)=\sum_{j=0}^n c_j(n)  e^{ix_1 h_{j,1}(n)}e^{ix_2 h_{j,2}(n)}\ldots e^{ix_d h_{j,d}(n)},
\end{equation}
where  $(c_j(n))_{j,n}$, $j=0,\ldots ,n$, for $ n\in \mathbb{N}_0$ is a  complex-valued sequence.
\end{definition}

\begin{definition}[Superoscillating sequence]\label{superoscill}
A generalized Fourier sequence  in several variables $F_n(x_1,\ldots ,x_d)$, with $d\in \mathbb{N}$
such that $d\geq 2$,
is said to be {\em a superoscillating sequence} if
$$
\sup_{j=0,\ldots ,n,\ n\in\mathbb{N}_0} \  |h_{j,\ell}(n)|\leq 1 ,\ \ {\rm for} \ \ell=1,...,d,
$$
 and there exists a compact subset of $\mathbb{R}^d$, which will be called {\em a superoscillation set}, on which
$F_n(x_1,\ldots ,x_d)$ converges uniformly to $e^{ix_1 g_1}e^{ix_2 g_2}\ldots e^{ix_d g_d}$, where $|g_\ell|>1$ for  $\ell=1,\ldots ,d$.
\end{definition}

In the paper \cite{JFAA} we studied the function theory of superoscillating functions in several variables under the additional hypothesis that
there exist $r_\ell \in \mathbb{N}$, such that
\begin{equation}\label{restric}
p= r_1q_1+\ldots +r_dq_d.
\end{equation}
In that case,
we proved that for
 $p$, $q_\ell \in \mathbb{N}$, $\ell=1,\ldots ,d$ the function
$$
F_n(x,y_1,\ldots ,y_d)=\sum_{j=0}^nC_j(n,a)
e^{ix(1-2j/n)^p}e^{iy_1(1-2j/n)^{q_1}}\ldots e^{iy_d(1-2j/n)^{q_d}}
$$
 is superoscillating when $|a|>1$, where $C_j(n,a)$ are suitable coefficients.
 In the paper \cite{ACJSSST22}, we were able to remove the condition (\ref{restric}), while here we will
 show that it is possible to replace the functions $(1-2j/n)^p$ in the terms
 $e^{ix(1-2j/n)^p}$ with more general holomorphic functions. As we shall see, different function spaces are involved in the proofs according to the fact that the  holomorphic functions are entire or not.

\begin{theorem}[The general case of $d\geq 2$ variables]\label{propm=2}
Let $d$ be a positive integer and let $R_\ell\in\rr_{+}\cup\{\infty\}$ be such that $R_\ell\geq 1$ for any $\ell=1,\dots,\, d$.
Let $G_1,\dots,\, G_d$ be holomorphic functions whose series expansion at zero is given by
\begin{equation}\label{seriesG}
G_\ell(\lambda)=\sum_{m_\ell=0}^\infty g_{\ell,m}\lambda^{m_\ell},\quad \forall \ell=1,\dots,\, d
\end{equation}
and, moreover, the sequences $(g_{\ell,m})$ satisfy the condition
$$
\lim \sup_{m\to\infty}|g_{\ell,m}|^{1/m}=\frac 1{R_\ell},\ \ \forall  \ell=1,\dots,\, d.
$$
Let
\begin{equation}\label{basic_sequence_bis}
f_n(x):= \sum_{j=0}^n Z_j(n,a)e^{ih_j(n)x},\ \ \ n\in \mathbb{N},\ \ \ x\in \mathbb{R},
\end{equation}
be superoscillating functions as in Definition \ref{SUPOSONE} and
assume that their entire extensions to the functions $f_n(\xi)$ converge to $e^{ia\xi}$ in $A_{1,B}$ for some positive real value $0<B<\frac R{4e}$ where $R:=\min_{\ell=1,\ldots, d} R_\ell$.
We define
$$
F_n(x_1,\ldots,x_d):=\sum_{j=0}^n Z_j(n,a)e^{ix_1 G_1(h_j(n))}e^{ix_2 G_2(h_j(n))} \ldots e^{ix_d G_d(h_j(n))}.
$$
Then, whenever $|a|<R$ we have
$$
\lim_{n\to \infty}F_n(x_1,x_2,\ldots,x_d)=e^{ix_1 G_1(a) }e^{ix_2 G_2(a) }\ldots e^{ix_d G_d(a) },
$$
uniformly on compact subsets of $\rr^d$. In particular, $F_n(x_1,x_2,\ldots,x_d)$ is superoscillating when $|a|>1$.
\end{theorem}
\begin{proof}
Since $R_\ell\geq 1$ for any $\ell=1,\ldots, d$ and $|h_j(n)|<1$, using \eqref{cp} we have the chain of equalities
\[
\begin{split}
& F_n(x_1,x_1,\ldots,x_d)=\sum_{j=0}^nZ_j(n,a)e^{ix_1 G_1(h_j(n))+ix_2 G_2(h_j(n))+ \ldots +ix_d G_d(h_j(n))}
\\
&
=\sum_{j=0}^nZ_j(n,a) \,
\sum_{m=0}^\infty \frac{1}{m!} \Big[ ix_1 G_1(h_j(n))+ix_2 G_2(h_j(n))+ \ldots +ix_d G_d(h_j(n))\Big]^m
\\
&
=\sum_{j=0}^nZ_j(n,a) \,
\sum_{m=0}^\infty \frac{1}{m!} \Big[ ix_1 \sum_{p=1}^{\infty}g_{1,p}(h_j(n))^p+ \ldots +ix_d \sum_{p=1}^{\infty}g_{d,p}(h_j(n))^p\Big]^m
\\
&
=\sum_{j=0}^nZ_j(n,a) \,
\sum_{m=0}^\infty \frac{1}{m!} \Big[ \sum_{p=1}^{\infty}(i x_1g_{1,p}+ \ldots +ix_dg_{d,p})(h_j(n))^p\Big]^m
\\
&
= \sum_{m=0}^\infty \frac{1}{m!}\sum_{k_1=0}^{\infty}\left(\sum_{k_2=0}^{k_1}\ldots \sum_{k_m=0}^{k_{m-1}} y_{k_m}y_{k_{m-1}-k_m}\ldots y_{k_1-k_2}\right) (h_j(n))^{k_1},
\end{split}
\]
where we have set
$$
y_p:=ix_1g_{1,p}+\ldots+ix_dg_{d,p},\ \ {\rm for} \ \ p=1,\ldots r \ \ {\rm with} \ \ r \in \mathbb{N}.
$$
We define the infinite order differential operator
\begin{equation}\label{arielm}
\mathcal{U}(x_1,x_2,\ldots, x_d,D_\xi):=\sum_{m=0}^\infty \frac{1}{m!}\sum_{k_1=0}^{\infty}\left(\sum_{k_2=0}^{k_1}\ldots \sum_{k_m=0}^{k_{m-1}} y_{k_m}y_{k_{m-1}-k_m}\ldots y_{k_1-k_2}\right) \frac {D_{\xi}^{k_1}}{i^{k_1}}.
\end{equation}
Since $0<B<\frac R{4e}$, Proposition \ref{CONTU_bis} implies that the operator
$$
\mathcal{U}(x_1,x_2,\ldots, x_d,D_\xi):A_{1,B}\mapsto A_{1, 8eB}
$$
is continuous. We observe that
\[
\begin{split}
F_n(x_1,x_2,\ldots,x_d)
=\mathcal{U}(x_1,x_2,\ldots, x_d,D_\xi)\sum_{j=0}^nZ_j(n,a)e^{i\xi h_j(n)}\Big|_{\xi=0}
\end{split}
\]
The explicit computation of the term $\mathcal{U}(x_1,\ldots , x_d,D_\xi) e^{i\xi a}$ gives
\[
\begin{split}
&\mathcal{U}(x_1,\ldots, x_d, D_\xi)e^{i\xi a}=\\
&=\sum_{m=0}^\infty \frac{1}{m!}\sum_{k_1=0}^{\infty}\left(\sum_{k_2=0}^{k_1}\ldots \sum_{k_m=0}^{k_{m-1}} y_{k_m}y_{k_{m-1}-k_m}\ldots y_{k_1-k_2}\right) \frac {D_{\xi}^{k_1}}{i^{k_1}}e^{i\xi a}
\\
&
=\sum_{m=0}^\infty \frac{1}{m!}\sum_{k_1=0}^{\infty}\left(\sum_{k_2=0}^{k_1}\ldots \sum_{k_m=0}^{k_{m-1}} y_{k_m}y_{k_{m-1}-k_m}\ldots y_{k_1-k_2}\right) a^{k_1}e^{i\xi a},
\end{split}
\]
so we finally get
\[
\begin{split}
&\lim_{n\to \infty} F_n(x_1,\ldots, x_d)=\\
&=\sum_{m=0}^\infty \frac{1}{m!}\sum_{k_1=0}^{\infty}\left(\sum_{k_2=0}^{k_1}\ldots \sum_{k_m=0}^{k_{m-1}} y_{k_m}y_{k_{m-1}-k_m}\ldots y_{k_1-k_2}\right) a^{k_1}e^{i\xi a}\Big|_{\xi=0}
\\
&
=\sum_{m=0}^\infty \frac{1}{m!}\sum_{k_1=0}^{\infty}\left(\sum_{k_2=0}^{k_1}\ldots \sum_{k_m=0}^{k_{m-1}} (y_{k_m}a^{k_m})(y_{k_{m-1}-k_m}a^{k_{m-1}-k_m})\ldots (y_{k_1-k_2}a^{k_1-k_2})\right)
\\
&=\sum_{m=0}^\infty \frac{1}{m!}
\left(\sum_{p=1}^\infty y_p a^p\right)^m  =\sum_{m=0}^\infty \frac{1}{m!}
\left(ix_1G_1(a)+\ldots+ix_dG_d(a)\right)^m = e^{ix_1G_1(a)+\ldots+ ix_dG_d(a)}
\end{split}
\]
where the third equality is due to the formula \eqref{cp} and the fourth equality holds because we are assuming $|a|<R$. The previous limit is uniform over the compact subset of $\rr^d$ because of Remark \ref{supernova}.

\end{proof}

\begin{remark}
From the inspection of the proof we observe that:
\\
(I) The space of the entire functions on which the infinite order differential operator $\mathcal{U}(x_1,\ldots ,\,x_d, D_\xi)$ acts is the space $A_{1,B}$ in one complex variable, for some positive real value $0<B<\frac R{4e}$.
\\
(III) The variables $(x_1,x_1,\ldots,x_d)$ become the coefficients of the infinite order differential operator $\mathcal{U}(x_1,x_2,\ldots, x_d,D_\xi)$, defined in (\ref{arielm}), that still acts
 on the space $A_{1,B}$.
\end{remark}

\section{Supershifts in several variables}

The procedure to define superoscillating functions can be extended to the case of supershift.
 Recall that the supershift property of a function extends the notion
 of superoscillation and that this concept, that we recall below, turned out to be a crucial ingredient for the study of the evolution
 of superoscillatory functions as initial conditions of the Schr\"odinger equation.

\begin{definition}[Supershift]\label{Super-shift}
Let $\mathcal{I}\subseteq\mathbb{R}$ be an interval with $[-1,1]\subset \mathcal I$
and let
$\varphi:\, \mathcal I  \times \mathbb{R}\to \mathbb R$, be a continuous function on $\mathcal I$.
We set
$$
\varphi_{h}(x):=\varphi(h,x), \ \ h\in \mathcal{I},\ \ x\in \mathbb R
 $$
and we consider a sequence of points $(h_{j}(n))$ such that
 $$
  h_{j}(n)\in [-1,1] \ \  {\rm for} \ \  j=0,...,n \ \ {\rm  and} \ \ n\in\mathbb{N}_0.
 $$
   We define the functions
\begin{equation}\label{psisuprform}
\psi_n(x)=\sum_{j=0}^nc_j(n)\varphi_{h_{j}(n)}(x),
\end{equation}
where  $(c_j(n))$ is a sequence of complex numbers for $j=0,...,n$ and $n\in\mathbb{N}_0$.
If
$$
\lim_{n\to\infty}\psi_n(x)=\varphi_{a}(x)
$$
for some $a\in\mathcal I$ with $|a|>1$, we say that the function
$\psi_n(x)$, for  $x\in \mathbb{R}$, admits a supershift.
 \end{definition}

 \begin{remark}
 The term supershift comes from the fact that the interval
  $\mathcal I$ can be arbitrarily large (it can also be $\mathbb R$)
  and that the constant $a$ can be arbitrarily far away from the interval $[-1,1]$ where the functions $\varphi_{h_{j,n}}(\cdot)$ are indexed, see \eqref{psisuprform}.
 \end{remark}

\noindent Problem \ref{gsgdfaA}, for the supershift case, is formulated as follows.

\begin{problem}\label{gsgdfaSUP}
Let $h_j(n)$ be a given set of points in $[-1,1]$, $j=0,1,...,n$, for $n\in \mathbb{N}$ and let $a\in \mathbb{R}$ be such that $|a|>1$.
Suppose that for every $x\in \mathbb{R}$ the function $h\mapsto G(h x )$
extends to a holomorphic and entire function in $h$.
Consider  the functions
$$
f_n(x)=\sum_{j=0}^nY_j(n,a)G(h_j(n)x),\ \ \ x\in \mathbb{R}
$$
where
$h\mapsto G(h x)$ depends on the parameter $x\in \mathbb{R}$.
Determine the coefficients $Y_j(n)$
in such a way that
\begin{equation}\label{eqsuper}
f_n^{(p)}(0)=(a)^pG^{(p)}(0)\ \ \ for \ \ \ p=0,1,...,n.
\end{equation}
\end{problem}
The solution of Problem \ref{gsgdfaSUP}, obtained in \cite{newmet}, is summarized in the following theorem.
\begin{theorem}\label{thsupsh}
Let $h_j(n)$ be a given set of points in $[-1,1]$, $j=0,1,...,n$ for $n\in \mathbb{N}$ and let $a\in \mathbb{R}$ be such that $|a|>1$.
If $h_j(n)\not= h_i(n)$ for every $i\not=j$ and $G^{(p)}(0)\not=0$ for all $p=0,1,...,n$,
then there exists a unique solution $Y_j(n,a)$ of the linear system (\ref{eqsuper})
and it is given by
$$
Y_j(n,a)=
\prod_{k=0,\  k\not=j}^n\Big(\frac{h_k(n)-a}{h_k(n)-h_j(n)}\Big),
$$
so that
$$
f_n(x)=\sum_{j=0}^n\prod_{k=0,\  k\not=j}^n\Big(\frac{h_k(n)-a}{h_k(n)-h_j(n)}\Big)G(h_j(n)x),\ \ \ x\in \mathbb{R}.
$$
\end{theorem}
\begin{remark}
In the sequel, we shall move from the real to the complex setting and we will consider those functions $G$ and sequences
 $h_j(n)$ for which the holomorphic extension $f_n(z)$ of $f_n(x)$ converges in $A_1$ to $G(az)$.
\end{remark}

We can now extend the notion of supershift of a function in several variables.

\begin{definition}[Supershifts in several variables]
Let $|a|>1$.
For $d\in \mathbb{N}$ with $d\geq 2$, we assume that
$(x_1,...,x_d)\in \mathbb{R}^d$. Let
 $(h_{j,\ell}(n))$,  $j=0,...,n$  for $ n\in \mathbb{N}_0$, be  real-valued sequences for $\ell=1,...,d$
 such that for
$$
\sup_{j=0,\ldots ,n,\ n\in\mathbb{N}_0} \  |h_{j,\ell}(n)|\leq 1 ,\ \ {\rm for} \ \ell=1,...,d
$$
and let $G_\ell(\lambda)$, for $ \ell=1,...,d$, be entire holomorphic functions.
We say that
the sequence
\begin{equation}\label{basic_sequence_sevFF}
F_n(x_1,\ldots ,x_d)=\sum_{j=0}^n c_j(n)
G_1(x_1 h_{j,1}(n)) G_2(x_2 h_{j,2}(n))\ldots G_d(x_d h_{j,d}(n)),
\end{equation}
where  $(c_j(n))_{j,n}$, $j=0,\ldots ,n$, for $ n\in \mathbb{N}_0$ is a  complex-valued sequence,
admits the supershift property if
$$
\lim_{n\to \infty}F_n(x_1,\ldots ,x_d)=
G_1(x_1 a) G_2(x_2 a)\ldots G_d(x_d a).
$$

\end{definition}

\begin{theorem}[The case of $d \geq 1$ variables]\label{propm=3}
Let $|a|>1$ and let
\begin{equation}\label{basic_sequence}
f_n(x):= \sum_{j=0}^n Z_j(n,a)e^{ih_j(n)x},\ \ \ n\in \mathbb{N},\ \ \ x\in \mathbb{R},
\end{equation}
be a superoscillating function as in Definition \ref{SUPOSONE} and
assume that its holomorphic extension to the entire functions $f_n(z)$ converges to $e^{iaz}$ in the space $A_{1,B}$ for some positive real value $B$.
Let $d$ be a positive integer and let $R_\ell\in\rr_{+}\cup\{\infty\}$ for any $\ell=1,\dots,\, d$. Let $G_1,\dots,\, G_d$ be holomorphic functions whose series expansion at zero is given by
\begin{equation}\label{seriesG}
G_\ell(\lambda)=\sum_{m_\ell=0}^\infty g_{\ell,m}\lambda^{m_\ell},\quad \forall \ell=1,\dots,\, d.
\end{equation}
Moreover, we suppose the sequences $(g_{l,m})$'s satisfy the condition
$$
\lim \sup_{m\to\infty}|g_{\ell,m}|^{1/m}=\frac 1{R_\ell},\ \ \forall  \ell=1,\dots,\, d.
$$
We define
$$
F_n(x_1,\dots ,x_d)=\sum_{j=0}^nZ_j(n,a)G_1(x_1 h_j(n))\cdots G_d(x_dh_j(n)),
$$
where $Z_j(n,a)$ are given as in (\ref{basic_sequence}). Then, $F_n(x_1,\dots,\,x_d)$ admits the supershift property that is
$$
\lim_{n\to \infty}F_n(x_1,\dots, x_d)=G_1(x_1a)\cdots G_d(x_da)
$$
uniformly on compact subsets of $\{\underline x\in\mathbb R^d: | x_\ell|<R'\,\textrm{for any $\ell=1,\ldots, d$}\}$ where
\[
R':=\min\left( \frac R{|a|},\, \frac{R}{4eB} ,\,R \right)
\qquad
{\it \ where\ \  }
R:=\min_{\ell=1,\dots, d}R_\ell.
\]
\end{theorem}

\begin{proof}
Since $| x_\ell|<R\,\,\textrm{for any $\ell=1,\ldots, d$}$, we have
\[
\begin{split}
F_n(x_1,\dots, x_d)&=\sum_{j=0}^nZ_j(n,a)G_1(x_1 h_j(n))\dots G_d(x_d h_j(n))
\\
&
=\sum_{j=0}^nZ_j(n,a)\sum_{m_1=0}^\infty  g_{m_1}\dots \sum_{m_d=0}^\infty g_{m_d}x_1^{m_1}\cdots x_d^{m_d}(h_j(n))^{m_1+\dots +m_d}.
\end{split}
\]
 We now consider the
 auxiliary complex variable $\xi$ and we note that
\begin{equation}
\lambda^\ell=\frac{1}{i^\ell}D_\xi^\ell e^{i\xi \lambda}\Big|_{\xi=0}\ \ \ {\rm for }\ \ \ \lambda\in \mathbb{C}, \ \ \ \ell\in \mathbb{N},
\end{equation}
where $D_\xi$ is the derivative with respect to $\xi$ and $|_{\xi=0}$ denotes the restriction to $\xi=0$. We have
\[
\begin{split}
F_n(x_1,\ldots, x_d)&=\sum_{j=0}^nZ_j(n,a)\sum_{m_1=0}^\infty  g_{m_1}\dots \sum_{m_d=0}^\infty g_{m_d}x_1^{m_1}\cdots x_d^{m_d}[h_j(n)]^{m_1+\dots +m_d}
\\
&
= \sum_{j=0}^nZ_j(n,a)\sum_{m_1=0}^\infty  g_{m_1}\dots\sum_{m_d=0}^\infty g_{m_d}x_1^{m_1}\cdots x_d^{m_d} \frac{1}{i^{m_1+\dots +m_d}}D_\xi^{m_1+\dots +m_d} e^{i\xi h_j(n)}\Big|_{\xi=0}
\\
&
= \sum_{m_1=0}^\infty  g_{m_1}\dots \sum_{m_d=0}^\infty g_{m_d}x_1^{m_1}\cdots x_d^{m_d} \frac{1}{i^{m_1+\dots +m_d}}D_\xi^{m_1+\dots +m_d} \sum_{j=0}^nZ_j(n,a)e^{i\xi h_j(n)}\Big|_{\xi=0}.
\end{split}
\]
We define the operator
$$\mathcal{V}(x_1,\dots, x_d, D_\xi):= \sum_{m_1=0}^\infty  g_{m_1}\dots \sum_{m_d=0}^\infty g_{m_d}x_1^{m_1}\cdots x_d^{m_d} \frac{1}{i^{m_1+\dots +m_d}}D_\xi^{m_1+\dots +m_d}$$
so that we can write
$$
F_n(x_1,\dots, x_d)=\mathcal{V}(x_1,\dots, x_d, D_\xi)\sum_{j=0}^nZ_j(n,a)e^{i\xi h_j(n)}\Big|_{\xi=0}.
$$
Since $| x_\ell|<\frac R{4eB}\,\,\textrm{for any $\ell=1,\ldots, d$}$, we can use Proposition \ref{contsupshif_bis} in order to compute the following limit
\[
\begin{split}
\lim_{n\to \infty} F_n(x_1,\dots, x_d)& = \mathcal{V}(x_1,\dots, x_d, D_\xi) \lim_{n\to\infty} \sum_{j=0}^nZ_j(n,a)e^{i\xi h_j(n)}\Big|_{\xi=0}\\
& = \mathcal{V}(x_1,\dots, x_d, D_\xi) e^{i\xi a}\Big|_{\xi=0}\\
& = \sum_{m_1=0}^\infty  g_{1,m_1}\dots \sum_{m_d=0}^\infty g_{d,m_d}x_1^{m_1}\dots x_d^{m_2} \frac{1}{i^{m_1+\dots +m_d}}D_\xi^{m_1+\dots +m_2}e^{i\xi a}\Big|_{\xi=0}\\
&=\sum_{m_1=0}^\infty  g_{1,m_1}\dots \sum_{m_d=0}^\infty g_{d,m_d}(ax_1)^{m_1}\dots (ax_d)^{m_2}=G_1(ax_1)\cdots G_d(ax_d)
\end{split}
\]
where the last equality holds because we are assuming $| x_\ell|< \frac R{|a|}\,\,\textrm{for any $\ell=1,\ldots, d$}$. The previous limit is uniform over the compact subset of $\{\underline x\in\rr^d:\, |x_\ell|<R'\,\textrm{for any $\ell=1,\ldots, d$}\}$ because of Remark \ref{rnova1}.
\end{proof}

\begin{remark}
A special case of the previous theorem occurs when the holomorphic functions $G_\ell$'s are entire functions. Moreover, differently from Theorem \ref{propm=2}, in Theorem \ref{propm=3} the parameters $x_\ell$ appear in the arguments of the functions $G_\ell$'s. This implies that the hypothesis of Theorem \ref{propm=3} imposes more constraints on the parameters $x_\ell$'s, namely $ |x_\ell|<R'$ for any $\ell=1,\ldots, d$.
\end{remark}

\end{document}